\documentclass[a4paper,10pt]{article}
\usepackage[utf8]{inputenc}
\usepackage{amsmath}
\usepackage{amsthm}
\usepackage{amsfonts}
\usepackage{amssymb}
\usepackage{mathrsfs}
\usepackage{geometry}

\title{Patterns in numbers and infinite sums and products}
\author{YINING HU \\
CNRS, Institut de Math\'ematiques de Jussieu-PRG \\
Universit\'e Pierre et Marie Curie, Case 247 \\
4 Place Jussieu \\
F-75252 Paris Cedex 05 (France) \\
{\tt yining.hu@imj-prg.fr}}
\date{}
\begin{document}

\maketitle

\begin{abstract}
Let $a_{w,B}(n)$ denote the the number of occurences of the word $w$ in the base $B$ expansion of the non-negative integer $n$. In this article we
generalize the results of Allouche and \mbox{Shallit \cite{infprod}} by
proving the existence of a finite set $L_{w,B}$ of pairs $(l,c_l)$ where $l$ 
is a polynomial with integer coefficients of degree 1  and $c_l$ an integer such that:
$$\sum\limits_{n\geq 0} (-1)^{a_{w,B}(n)}\sum\limits_{(l,c_l)\in L_{w,B}} c_l f(l(n))=
\begin{cases}0 &\mbox{ if } w\neq 0^j,\\
 -2\cdot(-1)^{a_{w,B}(0)}f(0) &\mbox{ if } w=0^j
\end{cases}$$
where $f$ is any function that verifies certain convergence conditions.

After exponentiating, we recover previous results and obtain new ones such as
$$\prod \limits_{n\geq 1} \left( \frac{3n+1}{3n+2}\right) ^{(-1)^{n}}=\frac{2}{\sqrt{3}},$$
and
$$\prod \limits_{n\geq 1} \left( \frac{9n+7}{9n+8} \right)^{(-1)^{a_{21,3}(n)}}=\frac{8}{7\sqrt{3}}. $$
\end{abstract}
\newtheorem{coro}{Corollary}
\section{Introduction}
Let $s_q(n)$ denote the sum of digits of the non-negative integer $n$ when written in base $q$. Woods and Robbins \cite{robbins,woods}
proved that 
\begin{equation}\label{eq:rw}
 \prod\limits_{n\geq 0} \bigg (\frac{2n+1}{2n+2} \bigg )^{(-1)^{s_2(n)}}=\frac{\sqrt{2}}{2}. 
\end{equation}
\newtheorem{thm}{Theorem}
\newtheorem{lemma}{Lemma}
Allouche and Shallit \cite{infprod} looked at the function $a_w(n)$, defined as the number of occurrences of the finite non-empty
binary word $w$ in the binary expansion of $n$. With this notation the $s_2(n)$ in Equation \ref{eq:rw} becomes $(-1)^{a_1(n)}$. 
%
With the following two theorems, they generalized the result to $a_w(n)$ for all $w$.
\pagebreak

\begin{thm}[Allouche and Shallit \cite{infprod}]\label{3}
Let $w$ be a string of zeros and ones, and
$$g=2^{|w|-1}, \;\;\; h=\lfloor v(w)/2\rfloor, $$
and let $X$ be a complex number with $|X|\leq1$ and $X\neq 1$. Then
$$\sum\limits_{n}X^{a_w(gn+h)}L(2gn+v(w))=-\frac{1}{1-X},$$
where the sum is over $n\geq 1$ for $w=0^j$ and $n\geq 0$ otherwise.
\end{thm}

\begin{thm}[Allouche and Shallit \cite{infprod}]\label{10}
 There is an effectively computable rational function $b_w(n)$ such that, for all $X\neq 1$ with $|X|\leq 1$, we have
 \begin{equation}\label{eq:b}
 \sum\limits_{n}\log_2(b_w(n))X^{a_w(n)}=-\frac{1}{1-X}, 
 \end{equation}
  where the sum is over $n\geq 1$ for $w=0^j$ and $n\geq 0$ otherwise.
\end{thm}

By setting $X=-1,\;w=1$ in equation \ref{eq:b} and exponentiating we rediscover equation \ref{eq:rw}. Other values of $w$ give new results; 
for example, 
\begin{equation}
 \prod\limits_{n\geq 0} \bigg (\frac{(4n+2)(8n+7)(8n+3)(16n+10)}{(4n+3)(8n+6)(8n+2)(16n+11)} \bigg )^{(-1)^{a_{1010}(n)}}=\frac{\sqrt{2}}{2}. 
\end{equation}

In this article we generalize the results in \cite{infprod}: our result applies to any base and all functions that verify certain 
convergence conditions.

\section{Notation}
We let $\mathbb{N}$ denote the set of non-negative integers. Let $B$ be an integer greater than 1. Let $w$ be a finite non-empty word over
$\{0,...,B-1\}$ (that is, $w\in \{0,...,B-1\}^* $).
Let $v_B:\{0,...,B-1\}^* \rightarrow \mathbb{N}$ be the map that assigns to $w$ its value when interpreted in base $B$. 
For example, $v_2(110)=6$. Let $|w|$
denote the length of $w$. For $x\in \{0,...,B-1\}$, let $\hat{x}$ denote $x+1\;\textnormal{mod } B$.

Let $a_{w,B}(n)$ count the number of (possibly overlapping) occurrences of the block $w$ in the expansion
of $n$ in base $B$. For example, $a_{22,3}(26)=2$. We use the same convention as in \cite{hurwitz} in the case where $w$ starts with a zero; if $w\neq 0^j$, 
then in evaluating $a_{w,B}(n)$ we assume that the expansion of $n$ starts with an arbitrarily long prefix of zeros. 
Thus $a_{011,2}(6)=1$. If $w=0^j$, we use the expansion of $n$ which starts with a non-zero digit. This means in particular that $a_{0,B}(0)=0$. 

To simplify notation, we write $a(n)$ instead of $(-1)^{a_{w,B}(n)}$ when there is no confusion.

\section{The main result}\label{sec:main}
Our goal in this section is to prove the existence of a finite set $L_{w,B}$ of pairs $(l,c_l)$ where $l$ 
is a first degree integer coefficient polynomial and $c_l$ an integer such that:
$$\sum\limits_{n\geq 0} (-1)^{a_{w,B}(n)}\sum\limits_{(l,c_l)\in L_{w,B}} c_l f(l(n))=
\begin{cases}0 &\mbox{ if } w\neq 0^j,\\
 -2a(0)f0) &\mbox{ if } w=0^j
\end{cases}$$
where $f$ is any function verifying certain convergence conditions that will be made precise later.

First we note the following proposition which will be proved in Section \ref{sec:conv}.
\newtheorem{prop}{Proposition}
\begin{prop}
Let $S(n)=\sum\limits_{k=0}^{n-1} a(k)$, where $a(n)=(-1)^{a_{w,B}(n)}$ and $w$ is a non-empty word over $\{0,\ldots,B-1\}$ of length $k$. Then 
$$|S(n)|\begin{cases}

=1 \mbox{ or } 0 &\mbox{if } B=2 \mbox{ and } k=1\\
\leq 2\cdot \lceil \log_3 (n) \rceil &\mbox{if } B=3 \mbox{ and } k=1\\ 
=O(n^{\alpha}) \mbox{ where } \alpha=\log_{V}(V-2)<1 \mbox{ and }\; V=B^k &\mbox{otherwise}.

       \end{cases}
$$
\end{prop}

\theoremstyle{remark}
\newtheorem*{remark}{Remark}
\begin{remark}
 By Theorem 3.1 in \cite{regular} we know that $S(n)$ is a regular sequence, and Theorem 2.10 from the same article confirms that there exists
 a constant $c$ such that $S(n)=O(n^c)$.
\end{remark}

The following lemma is inspired by the general lemma in \cite{hurwitz}.

\begin{lemma}\label{lm:main}
 Let $B$ be an integer, $B\geq 2$, let $w$ be a word over $\{0,\ldots,B-1\}$ ending in the symbol $e$ 
 and let $f:\mathbb{N}\rightarrow \mathbb{C}$ be a function such that 
 $f(n)=O(n^{\beta})$ and $f(n+1)-f(n)=O(n^{\beta-1})$ for $\beta<0$ if $B=2$ or $3$ and $k=1$, and  $\beta <-\alpha$ otherwise. 
  Then
 \begin{equation}\label{eq:main}\sum\limits_{n\geq 0} a(n)(f(n)-\sum\limits_{j=0}^{B-1}f(Bn+j))
  =2\sum\limits_{m}a(B^{|w|}m+v_B(w))f(B^{|w|}m+v_B(w)),
 \end{equation}
 where the last summation is taken over $m\geq 0$, except when $w=0^j$, where it is taken over $m\geq 1$.

\end{lemma}
\begin{proof}
The convergence of $\sum a(n)f(n)$ and $ \sum a(n)f(Bn+j)$ is assured by Corollary \ref{coro} in \mbox{Section \ref{sec:conv}}.
Let $e$ belong to $\{0,\ldots,B-1\}$.
\begin{align*}
&\sum\limits_{n\geq 0} a(n) (f(n)-\sum\limits_{j=0} ^{B-1} f(Bn+j))
\\=& \sum\limits_{n\geq 0} \sum\limits_{j=0} ^{B-1} a(Bn+j)f(Bn+j)-\sum\limits_{n\geq 0} \sum\limits_{j=0} ^{B-1} a(n)f(Bn+j) 
\\=& \sum\limits_{n\geq 0}  a(Bn+e)f(Bn+e)-\sum\limits_{n\geq 0} a(n)f(Bn+e).
\end{align*}
If $w\neq 0^j$, one has:
$$a(Bn+e)=\begin{cases}
           -a(n) &\mbox{if } \exists m\geq 0 \mbox{ such that } n=B^{|w|-1}m+\lfloor \frac{v_B(w)}{B}\rfloor\\
           a(n) &\mbox{otherwise}.
          \end{cases}
$$
If $w= 0^j$, one has:
$$a(Bn+e)=\begin{cases}
           -a(n) &\mbox{if } \exists m \geq 1 \mbox{ such that } n=B^{|w|-1}m+\lfloor \frac{v_B(w)}{B}\rfloor\\
           a(n) &\mbox{otherwise}.
          \end{cases}
$$
Hence $$\sum\limits_{n\geq 0} a(n)(f(n)-\sum\limits_{j=0}^{B-1}f(Bn+j))=2\sum\limits_{m}a(B^{|w|}m+v_B(w))f(B^{|w|}m+v_B(w)),$$
 where the last summation is taken over $m\geq 0$, except when $w=0^j$, where it is taken over $m\geq 1$.

\end{proof}

For example, for $w=11$ and $B=2$, letting $a(n)$ denote $(-1)^{a_{11,2}(n)}$, we find
\begin{equation}\label{eq:sha}
\sum\limits_{n\geq 0} a(n)(f(n)-f(2n)-f(2n+1))=2\sum\limits_{m\geq 0}a(4m+3)f(4m+3). \end{equation}

The next step consists of transforming the sum with $a(4m+3)$ on the right to a sum with $a(m)$. First, noticing that
$\{4m+3|m\in\mathbb{N} \}=\{2m+1|m\in\mathbb{N} \}\backslash \{4m+1|m\in\mathbb{N} \}$, we split the sum into two sums. 
Then, we replace $a(4m+1)$ by $a(m)$, as $01$ is not a suffix of $11$. And we continue like this:

\begin{align*}
 \sum\limits_{m\geq 0}a(4m+3)f(4m+3)&= \sum\limits_{m\geq 0}a(2m+1)f(2m+1)- \sum\limits_{m\geq 0}a(4m+1)f(4m+1)\\
 &= \sum\limits_{m\geq 0}a(2m+1)f(2m+1)- \sum\limits_{m\geq 0}a(m)f(4m+1)\\
 &=\sum\limits_{m\geq 0}a(m)f(m)-\sum\limits_{m\geq 0}a(2m)f(2m)- \sum\limits_{m\geq 0}a(m)f(4m+1)\\
 &=\sum\limits_{m\geq 0}a(m)(f(m)-f(2m)-f(4m+1)).
\end{align*}

Substituting this in Equation \ref{eq:sha}, we get
\begin{equation}\label{eq:shap}
\sum\limits_{n\geq 0}a(n)(-f(n)+f(2n)-f(2n+1)+2f(4n+1))=0.
\end{equation}

The following lemma describes each step of the process in detail:

\begin{lemma}\label{it}
 Let $w$ be a non-empty word over $\{0,\ldots,B-1\}$ of length $k$, $s$ and $t$ positive integers such that $s\leq t \leq |w|$, $x$ an integer,
 and $m$ an integer whose base $B$ expansion is $b_1b_2...b_k$,
with possible leading zeros.

(A) If $b_1b_2...b_s$ is not a suffix of $w$, then
$$ \sum\limits_{n\geq 0} a(B^{s}n+m)f(B^{t}n+x) = \sum\limits_{n\geq 0} a(B^{s-1}n+v_B(b_1...b_{s-1}))f(B^{t}n+x).$$
 
(B) If $b_1b_2...b_s$ is a suffix of $w$, then
\begin{align*}\sum\limits_{n\geq 0} a(B^{s}n+m)f(B^{t}n+x)= \sum\limits_{n\geq 0} a(B^{s-1}n+v_B(b_2...b_s))f(B^{t-1}n+x-B^{t-1}b_1)
\\ -\sum\limits_{b\in \{0,\ldots,B-1\}\backslash \{b_1\}}\sum\limits_{n\geq 0} a(B^{s-1}n+v_B(bb_2...b_{s-1}))f(B^{t}n+x+B^{t-1}(b-b_1)).
\end{align*}
\end{lemma}
\begin{proof}
 To prove (A), we only need to note that if $b_1b_2...b_s$ is not a suffix of $w$, then $a_{w,B}(B^{s}+b_1b_2...b_s)=
 a_{w,B}(B^{s-1}+b_1b_2...b_{s-1})$.\\
Now suppose that $b_1b_2...b_s$ is a suffix of $w$:
\begin{align*}
 & \{ (B^{s-1}n+v_B(b_2...b_s),B^{t-1}n+x-B^{t-1}b_1)| n\in \mathbb{N} \} \\
 =& \bigcup_{ b\in \{0,\ldots,B-1\}} \{B^{s-1}(Bn+b)+v_B(b_2...b_s),B^{t-1}(Bn+b)+x-B^{t-1}b_1)| n\in \mathbb{N} \}\\
 =&  \{(B^{s-1}(Bn+b_1)+v_B(b_2...b_s),B^{t-1}(Bn+b_1)+x-B^{t-1}b_1)| n\in \mathbb{N} \} \cup \\
 &  \bigcup_{ b\in \{0,\ldots,B-1\}\backslash \{b_1\}} \{(B^{s-1}(Bn+b)+v_B(b_2...b_s),B^{t-1}(Bn+b)+x-B^{t-1}b_1)| n\in \mathbb{N} \}\\
 =&\{ (B^{s}n+v_B(b_1...b_s), B^t n+x)|n\in \mathbb{N} \}\cup \\
  &\bigcup_{ b\in \{0,\ldots,B-1\}\backslash \{b_1\}} \{(B^{s}n+v_B(bb_2...b_s),B^{t}n+x+B^{t-1}(b-b_1))| n\in \mathbb{N} \}.
\end{align*}
 As  $b_1b_2...b_s$ is a suffix of $w$, $bb_2...b_s$ cannot be a suffix of $w$ for $ b\in \{0,\ldots,B-1\}\backslash \{b_1\}$, therefore
 $a(B^{s}n+v_B(bb_2...b_s))=a(B^{s-1}n+v_B(bb_2...b_{s-1}))$ for $b\neq b_1$, which proves (B).
\end{proof}

Iterating the process above to the sum on the right of Equation \ref{eq:main} gives us the desired result:

\begin{thm}\label{th:main}
Let $f:\mathbb{N}\rightarrow \mathbb{C}$ be a function such that 
 $f(n)=O(n^{\beta})$ and $f(n+1)-f(n)=O(n^{\beta-1})$ for $\beta<0$ if $B=2$ or $3$ and $k=1$, and  $\beta <-\alpha$ otherwise. 
 There exists a a finite set $L_{w,B}$ of pairs $(l,c_l)$ where $l$ 
is a polynomial with integer coefficients of degree 1 and $c_l$ an integer such that:

$$\sum\limits_{n\geq 0} (-1)^{a_{w,B}(n)}\sum\limits_{(l,c_l)\in L_{w,B}} c_l f(l(n))=
\begin{cases}0 &\mbox{ if } w\neq 0^j,\\
 -2a(0)f(0) &\mbox{ if } w=0^j.
\end{cases}$$

\end{thm}
\begin{proof}
 First we rewrite Lemma \ref{lm:main} as 
 \begin{align*}&\sum\limits_{n\geq 0}a(n)(f(n)-\sum\limits_{j=0}^{B-1}f(Bn+j))-2\sum\limits_{m\geq 0}a(B^{|w|}m+v_B(m))f(B^{|w|}m+v_B(m))
 \\&=
 \begin{cases} 0 &\mbox{if } w \neq 0^j \\ 
  -2a(0)f(0) &\mbox{if } w=0^j
 \end{cases}.\end{align*}
Then we successively apply Lemma \ref{it} to 
 $\sum\limits_{m\geq 0}a(B^{|w|}m+v_B(w))f(B^{|w|}m+v_B(w))$. We verify easily that in both cases of Lemma \ref{it},
a sum with $a(B^sn+m)$ either becomes one sum or the sum of $B$ sums with $a(B^{s-1}n+v_B(v))$ where $v$ is an appropriate word 
of length $s-1$. After each iteration,
the new sums still verify the condition of Lemma \ref{it}. In $|w|$ steps we will have only sums of the form $\sum a(n)f(l(n))$,
where $l$ is a polynomial with integer coefficients of degree 1.
\end{proof}
It can be shown that the set $L_{w,B}$ is effectively computable using arguments similar to those found in \cite{algo}.

\begin{remark}
Given a function $g :\mathbb{N} \rightarrow \mathbb{C}$ that verifies the convergence condition of Corollary \ref{coro},
it is an interesting question to ask if there exists another function $f :\mathbb{N} \rightarrow \mathbb{C}$ such that
$$g(n)=-f(n)+f(2n)-f(2n+1)+2f(4n+1) \;\; \forall n\in \mathbb{N}. $$
It can be easily seen that there exists an infinity of choices for $f$. If we require $f$ to verify the convergence condition, 
this question becomes tricky. Equation \ref{eq:shap} tells us that such a function $f$ exists only if $g(0)=-\sum\limits_{n\geq 1}
(-1)^{a_{11,2}(n)}g(n)$, which is not evident to establish otherwise.
\end{remark}
 
\section{Link with previous results}

Theorem \ref{th:main} in this article contains the results in \cite{infprod} when $X$ in Theorem \ref{3} and Theorem \ref{10} is replaced by $-1$. 
In this section we first give an example, then we prove that the two methods always give the same identities where Theorem \ref{10} applies, that is,
for $B=2$, and $f(n)=L(n)$ with $L(n)=\log_2 (\frac{n}{n+1})$ if $n>0$ and $L(0)=0$.

Taking for example $w=11$, by Theorem \ref{th:main}, we have
$$\sum\limits_{n\geq 1} (-1)^{a_{11,2}(n)}\log_2 \bigg ( \frac{(2n+1)^2}{(n+1)(4n+1)}\bigg )=-\frac{1}{2}.$$
After exponentiating we find:
$$\prod\limits_{n\geq 1} \bigg ( \frac{(2n+1)^2}{(n+1)(4n+1)}\bigg )^{(-1)^{a_{11,2}}}=\frac{\sqrt{2}}{2}. $$
This can be obtained alternatively by substituting $-1$ for $X$ in  Theorem \ref{3} and applying susccessively Lemma 4 in \cite{infprod}. 

In fact, when we substitute $-1$ for $X$ 
in Theorem \ref{3}, we get
$$\sum\limits_n (-1)^{a_{w,2}(2^{|w|-1}n+\lfloor v_2(w)/2 \rfloor)} L(2^{|w|}n+v_2(w))=-\frac{1}{2}.$$
On the other hand, if we apply Lemma \ref{lm:main} to $f=L$, as 
$L(n)-L(2n)-L(2n+1)=\log_2(\frac{n}{n+1}\cdot \frac{2n+1}{2n}\cdot \frac{2n+2}{2n+1})=0 $ for $n\geq 1$,
the left side of the identity becomes
$$ -(-1)^{a_{w,2}(0)}L(1)+\sum\limits_{n\geq 1} (-1)^{a_{w,2}(n)}(L(n)-L(2n)-L(2n+1))=1, $$
and the right side,
$$ 2\sum\limits_{n} (-1)^{a_{w,2}(2^{|w|}+v_2(w))}f(2^{|w|}+v_2(w) =-2\sum\limits_{n} (-1)^{a_{w,2}(2^{|w|-1}+\lfloor v_2(w)/2 \rfloor)}f(2^{|w|}+v_2(w)).  $$
The identity in Lemma  \ref{lm:main}  becomes
$$\sum\limits_n (-1)^{a_{w,2}(2^{|w|-1}n+\lfloor v_2(w)/2 \rfloor)} L(2^{|w|}n+v_2{w})=-\frac{1}{2}.$$
This is why we always find the same result using the two methods when $B=2$ and $f=L$.

\section{Examples}
\theoremstyle{definition}
\newtheorem{example}{Example}
\begin{example}
Let $s(n)=(-1)^{a_{1,3}(n)}=(-1)^n$. By Lemma \ref{lm:main} we have
$$\sum\limits_{n\geq 0}s(n)(f(n)-f(3n)-f(3n+1)-f(3n+2))=2\sum\limits_{n\geq 0}s(3n+1)f(3n+1)=-2\sum\limits_{n\geq 0}s(n)f(3n+1)$$
Therefore,$$ \sum\limits_{n\geq 0}s(n)(f(n)-f(3n)+f(3n+1)-f(3n+2))=0. $$
Taking $f(n)=\frac{1}{n}$ for $n>0$ and $f(0)=0$, we get 
$$\sum\limits_{n\geq 1}(-1)^n \left( \frac{2}{3n}+\frac{1}{3n+1}-\frac{1}{3n+2}\right) =-\frac{1}{2}.$$
Taking $f(n)=L(n)$, exponentiating and taking the square root, we get
$$\prod \limits_{n\geq 1} \left( \frac{3n+1}{3n+2}\right) ^{(-1)^n}=\frac{2}{\sqrt{3}}.$$
Another way of obtaining the identity above can be found in \cite[Section 4.4]{beta}.
\end{example}

\begin{example}
We have proved in the previous section that if $B=2$, and $f(n)=L(n)$, we obtain the same identities as in \cite{infprod}. But unlike \cite{infprod}, our method applies
to bases other than $2$ as well. Taking $B=3$ and $t(n)=(-1)^{a_{21,3}(n)}$ we have
$$\prod \limits_{n\geq 1} \left( \frac{9n+7}{9n+8} \right)^{t(n)}=\frac{8}{7\sqrt{3}}. $$
This is obtained by applying Lemma \ref{lm:main} and exponentiating:
\begin{align*}&\sum\limits_{n\geq 0} t(n)(L(n)-L(3n)-L(3n+1)-L(3n+2))=  2\sum\limits_{n\geq 0} t(9n+7)L(9n+7)=-2\sum\limits_{n\geq 0} t(n)L(9n+7).\\
\Rightarrow &\sum\limits_{n\geq 0} t(n)(L(n)-L(3n)-L(3n+1)-L(3n+2)+2L(9n+7))= 0.\\
 \Rightarrow&\sum\limits_{n\geq 1} t(n)\log_2(\frac{9n+7}{9n+8})=\frac{1}{2}\sum\limits_{n\geq 1} t(n)(L(n)-L(3n)-L(3n+1)-L(3n+2)+2L(9n+7))\\
 &=\frac{1}{2}(L(1)+L(2)-2L(7))=3-\log_2(7)-\frac{1}{2}\log_2(3).
\end{align*}
\end{example}

\section{Convergence}\label{sec:conv}

It is proved in \cite{infprod} that for base $B=2$, $S(n)=\sum\limits_{k=0}^{n} a(k)=O(n^{\alpha})$ for some $\alpha<1$. 
In this section we give a proof of a similar result for all bases.

\begin{lemma}
Let $w=w_1...w_k$ be a non-empty word over $\{0,\ldots,B-1\}$ of length $k$. Let $u$ be a word over $\{0,\ldots,B-1\}$ of length $l$, then there exist 
words $v, v'$
over $\{0,\ldots,B-1\}$ of length $k$ such that
$$ \forall n,\;a(B^{k+l}n+B^l v_B(v) + v_B(u))=-a(B^{k+l}n+B^l v_B(v') + v_B(u)), $$
where $a(n)$ denotes $(-1)^{a_{w,B}(n)}$.
\end{lemma}
\begin{proof}
If no prefix of $u$ is a proper suffix of $w$, then we can take $v=w$, and 
$$v'_i=w_i \;\textnormal{for}\; i \neq k, \textnormal{and } v'_k=\hat {w_k} $$
Otherwise let $d$ be the length of the longest prefix of $u$ that is a proper suffix of $w$. We define $v$ and $v'$ as follows:
$$ v_i=\hat {w_1}\;\textnormal{for}\;i=1,...,d;\;v_i=w_{i-d}\;\textnormal{for}\; d<i\leq k, $$
$$v'_i=v^1_i \;\textnormal{for}\; i \neq k, \textnormal{and}\; v'_k=\hat {v_k}. $$
\end{proof}

\begin{lemma}\label{b}
Let $w$ be a non-empty word over $\{0,\ldots,B-1\}$ of length $k$ and $a(n)=(-1)^{a_{w,B}(n)}$.
 Let $b_0(n)=a(n)$, $b_{i}(n)=\sum\limits_{j=0}^{B^k-1}b_{i-1}(B^k n+j)=\sum\limits_{j=0}^{B^{ki}-1}a(B^{ki}n+j)$, 
 then $$|b_i(n)|\leq (B^k-2)^i\;\textnormal{for}\; i\geq 1 .$$
\end{lemma}
\begin{proof}
Let us prove a stronger assertion: for all $i\geq 1$, there exists a subset $S_i$ of $[0,B^{ki}-1]$ of cardinality $(B^k-2)^i$ such that
for all $n$, $b_i(n)=\sum\limits_{j\in S_i} a(B^{ki} n+j)$.

  For $i=1$, $b_1(n)=\sum\limits_{j=0}^{B^k-1}a(B^k n+j)$. By the previous lemma, there exist $j_1,\;j_2$ with $0\leq j_1,j_2<B^k$ such that for all $n$,
 $ a(B^k n+j_1)=-a(B^k n +j_2)$. Thus we can define $S_1$ as $[0,B^k-1]\backslash \{j_1, j_2\}$ and
 $$ b_1(n)=\sum\limits_{j\in S_1} a(B^k n +j).$$
 Suppose that the assertion has been proved for $i$, let us prove that it is also true for $i+1$.
 By the induction hypothesis,  $$b_i(n)=\sum\limits_{j\in S_i} a(B^{ki} n+j),\;\textnormal{where}\; S_i \;\textnormal{is a subset of}\; 
 [0,B^{ki}-1]\; \textnormal{of cardinality}\; (B^k-2)^i. $$
 Therefore
 \begin{align*}b_{i+1}(n)&=\sum\limits_{m=0}^{B^k-1}b_i(B^k n+m)
 \\&=\sum\limits_{m=0}^{B^k-1}\sum\limits_{j\in S_i} a(B^{ki}(B^k n+m)+j)\\
 &=\sum\limits_{j\in S_i}\sum\limits_{m=0}^{B^k-1}a(B^{ki+k}n+B^{ki} m + j). \end{align*}
Again, by the previous lemma, for each $j$ in the first sum there exist $m_{j,1},m_{j,2}$ such that for all $n$,
$a(B^{ki+k}n+B^{ki} m_{j,1} + j)=-a(B^{ki+k}n+B^{ki} m_{j,2} + j) $. Thus in the inner sum there are at most $B^k-2$ terms, which proves
the existence of a subset $S_{i+1}$ of $[0,B^{k(i+1)}-1]$ of cardinality $(B^k-2)^{(i+1)}$ such that 
$b_{i+1}(n)=\sum\limits_{j\in S_{i+1}}a(B^{k(i+1)}n+j)$.
\end{proof}

Before proving the Proposition stated at the beginning of Section \ref{sec:main}, we illustrate with an example the first step of the proof, which consists of 
decomposing $S(n)$ into blocks of $b_i$. 
Take $B=3$, $|w|=2$ and $n=200$. First we write $n$ in base $B^{|w|}=9$: $n=2 \cdot 9^2 + 4 \cdot 9^1 + 2 \cdot 9^0$. 
$S(n)=b_2(0)+b_2(1)+b_1(18)+b_1(19)+b_1(20)+b_1(21)+b_0(198)+b_0(199)$. By Lemma \ref{b}, $|S(n)| \leq 2 \cdot 7^2 + 4 \cdot 7^1 + 2\cdot 7^0.$

\begin{proof}[Proof of Proposition 1]
We write $n$ in base $V$: $n=n_f V^f+n_{f-1} V^{f-1}+...+n_1 V+n_0$, where $0\leq n_i<V$ for $i=0,...,f$ and $n_f\neq 0$. We have 
$n\geq V^f$, and therefore $f\leq \log_V (n)$. 
On the other hand, by the previous lemma, 
$$|S(n)|\leq n_f (V-2)^f+ n_{f-1} (V-2)^{f-1}+...+n_1 (V-2)+n_0.$$
If $B=2$ and $k=1$, then $V=0$ and $|S(n)|=n_0$. If $B=3$ and $k=1$, then $|S(n)|\leq n_f+...+n_0 \leq 2\cdot \lceil \log_3 (n) \rceil $. Otherwise, we have
\begin{align*}
|S(n)|&\leq (V-1)((V-2)^f+(V-2)^{f-1}+...+(V-2)+1)\\
& < (V-1)\cdot 2(V-2)^f\\
&\leq 2(V-1)(V-2)^{\log_V(n)}\\
&=2(V-1)n^{\log_V(V-2)}\\
&=O(n^{\log_V(V-2)}).
\end{align*}

\end{proof}

We recall that $a(n)$ denotes $(-1)^{a_{w,B}(n)}$ and $\alpha=\log_{B^{|w|}}(B^{|w|}-2) $ for $(B,|w|)\neq (2,1)$ or $(3,1)$.
\begin{coro}\label{coro}
 Let $f:\mathbb{N}\rightarrow \mathbb{C}$ be a function such that 
 $f(n)=O(n^{\beta})$ and $f(n+1)-f(n)=O(n^{\beta-1})$ for $\beta<0$ if $B=2$ or $3$ and $k=1$, and  $\beta <-\alpha$ otherwise, then the series $\sum a(n)f(n)$ converges.
\end{coro}
\begin{proof}
The result is immediate when we use the formula of summation by parts. We define
$$S(n)=\sum\limits_{k=0}^{n}a(k)\;\;\;\textnormal{ and }  T(n)=\sum\limits_{k=0}^{n}a(k)f(k).$$
Then
 $$T(n)=f(n)S(n)+\sum\limits_{k=0}^{n-1}S(k)(f(k+1)-f(k)).$$
 
 If $B=2$ and $|w|=1$, $T(n)$ converges as $S(n)=O(1)$.
 
  If $B=3$ and $|w|=1$, $T(n)$ converges as $S(n)=O(\log(n))$.
  
Otherwise $f(n)S(n)=O(n^{\alpha+\beta})$ and $S(k)(f(k+1)-f(k))=O(k^{\alpha+\beta-1}) $, and 
$\lim_{n\rightarrow \infty} T(n)$ exists.
\end{proof}

\end{document}